\documentclass[10pt]{article}
\usepackage{latexsym, amsmath, amssymb, amsthm}
\usepackage{authblk, setspace}
\usepackage{graphicx, tikz, cite}
\usepackage[top=1 in, bottom=1 in, left=1 in, right=1 in]{geometry}

\newtheorem{theorem}{Theorem}[section]
\newtheorem{lemma}[theorem]{Lemma}

\newtheorem{remark}[theorem]{Remark}

\numberwithin{equation}{section}
 
\def\p{\partial}  \def\ora{\overrightarrow}
		\def\ol{\overline}		\def\m{\mathbb}		
\def\O{\Omega}  \def\lam{\lambda}  \def\eps{\epsilon}  
	\def\wt{\widetilde}
\def\be{\begin{equation}}     \def\ee{\end{equation}}
\def\bes{\begin{equation*}}		\def\ees{\end{equation*}}

\title{Lower bounds for the blow-up time of the heat equation in convex domains with local nonlinear boundary conditions}
\author[a]{Xin Yang\thanks{Email: yang2x2@ucmail.uc.edu}}
\author[b]{Zhengfang Zhou\thanks{Email: zfzhou@math.msu.edu}}
\affil[a]{Department of Mathematical Sciences, University of Cincinati, Cincinnati, OH 45220, USA}
\affil[b]{Department of Mathematics, Michigan State University, East Lansing, MI 48824, USA}

\date{}

\begin{document}
\maketitle

\begin{abstract}
This paper studies the lower bound for the blow-up time $T^{*}$ of the heat equation $u_t=\Delta u$ in a bounded convex domain $\Omega$ in $\mathbb{R}^{N}(N\geq 2)$ with positive initial data $u_{0}$ and a local nonlinear Neumann boundary condition: the normal derivative $\partial u/\partial n=u^{q}$ on partial boundary $\Gamma_1\subseteq\partial\Omega$ for some $q>1$, while $\partial u/\partial n=0$ on the other part. For any $\alpha<\frac{1}{N-1}$, we obtain a lower bound for $T^{*}$ which is of order $|\Gamma_{1}|^{-\alpha}$ as $|\Gamma_{1}|\rightarrow 0^{+}$, where $|\Gamma_{1}|$ represents the surface area of $\Gamma_{1}$. As $|\Gamma_{1}|\rightarrow 0^{+}$, this result significantly improves the previous lower bound $\ln\big(|\Gamma_1|^{-1}\big)$ and is almost optimal in dimension $N=2$, since the existing upper bound is of order $|\Gamma_{1}|^{-1}$ as $|\Gamma_{1}|\rightarrow 0^{+}$. In addition, the optimal asymptotic order of the lower bound for $T^{*}$ on $q$ (as $q\rightarrow 1^{+}$) and on $M_{0}$ (as $M_{0}\rightarrow 0^{+}$) are obtained, where $M_{0}$ denotes the maximum of $u_{0}$. 
\end{abstract}

%Keywords: Blow-up time,\, Lower bound,\, Local nonlinear Neumann boundary condition 
%
%2000 \it{MSC}: 35B44,\, 35K60,\, 35K05,\, 35C15

\bigskip
\bigskip

\section{Introduction}  
\label{Sec, introduction}
\subsection{Problem and Results}
In this paper, $\Omega$ represents a bounded open subset in $\mathbb{R}^{N}$ ($N\geq 2$) with $C^{2}$ boundary $\partial\Omega$; $\Gamma_1$ and $\Gamma_2$ denote two disjoint relatively open subsets of $\p\O$ which satisfy $\Gamma_{1}\neq \emptyset$ and $\ol{\Gamma}_1\cup\ol{\Gamma}_2=\p\O$. Moreover, $\widetilde{\Gamma}\triangleq \ol{\Gamma}_1\cap\ol{\Gamma}_2$ is assumed to be $C^{1}$ when being regarded as $\p\Gamma_1$ or $\p\Gamma_2$. We study the following problem:
\be\label{Prob}
\left\{\begin{array}{lll}
u_{t}(x,t)=\Delta u(x,t) &\text{in}& \Omega\times (0,T],\\
\frac{\partial u}{\partial n}(x,t)=u^{q}(x,t) &\text{on}& \Gamma_1\times (0,T],\\
\frac{\partial u}{\partial n}(x,t)=0 &\text{on}& \Gamma_2\times (0,T],\\
u(x,0)=u_0(x) &\text{in}& \Omega,
\end{array}\right.\ee
where 
\be\label{assumption on prob}
q>1,\, u_0\in C^{1}(\ol{\O}),\, u_0(x)\geq 0,\, u_0(x)\not\equiv 0. \ee
The normal derivative on the boundary is understood in the classical way: for any $(x,t)\in\p\O\times(0,T]$,
\be\label{Normal deri def, classical}
\frac{\p u}{\p n}(x,t)\triangleq \lim_{h\rightarrow 0^{+}} \frac{u(x,t)-u(x-h\ora{n}(x),t)}{h},\ee
where $\ora{n}(x)$ denotes the exterior unit normal vector at $x$. $\p\O$ being $C^2$ ensures that $x-h\ora{n}(x)$ belongs to $\O$ when $h$ is positive and sufficiently small. 

Throughout this paper, we write
\be\label{initial max}
M_0= \max_{x\in\ol{\O}}u_0(x)\ee
and denote $M(t)$ to be the supremum of the solution $u$ to (\ref{Prob}) on $\ol{\O}\times[0,t]$:
\be\label{max function at time t}
M(t)=\sup_{(x,\tau)\in \ol{\O}\times[0,t]}u(x,\tau).\ee
$|\Gamma_{1}|$ represents the surface area of $\Gamma_{1}$, that is
\[|\Gamma_{1}|=\int_{\Gamma_{1}}\,dS(x),\]
where $dS(x)$ means the surface integral with respect to the variable $x$. $\Phi$ refers to the fundamental solution to the heat equation:
\be\label{fund soln of heat eq}
\Phi(x,t)=\frac{1}{(4\pi t)^{N/2}}\,\exp\Big(-\frac{|x|^2}{4t}\Big), \quad\forall\, (x,t)\in\m{R}^{N}\times(0,\infty).\ee
In addition, the constants $C=C(a,b\dots)$ and $C_{i}=C_{i}(a,b\dots)$ will always be positive and finite and depend only on the parameters $a,b\dots$. One should also note that $C$ and $C_{i}$ may stand for different constants in different places.

The recent paper \cite{YZ16} studied (\ref{Prob}) systematically and the motivation was the Space Shuttle Columbia disaster in 2003, we refer the reader to that paper for the detailed discussion of the background. As a summary of its conclusions, \cite{YZ16} first established the local existence and uniqueness theory for (\ref{Prob}) in the following sense: there exists $T>0$ such that there is a unique function $u$ in $C^{2,1}\big(\O\times(0,T]\big)\bigcap C\big(\overline{\O}\times[0,T]\big)$ which satisfies (\ref{Prob}) pointwisely and also satisfies 
\be\label{interface bdry deri}
\frac{\p u}{\p n}(x,t)=\frac{1}{2}\,u^{q}(x,t), \quad \forall\, (x,t)\in\wt{\Gamma}\times (0,T]. \ee
Moreover, it is shown that this unique solution $u$ becomes strictly positive as soon as $t>0$. We want to remark here that the solution constructed in \cite{YZ16} through the heat potential technique automatically satisfies (\ref{interface bdry deri}) due to a generalized jump relation. The purpose of imposing this additional restriction (\ref{interface bdry deri}) to the local solution is to ensure the uniqueness through the Hopf's lemma, it is not clear whether the uniqueness will still hold without this restriction. After the local existence and uniqueness theory was set up, it also studied the blow-up phenomenon of (\ref{Prob}). If $T^{*}$ denotes the maximal existence time of the local solution $u$, then it is proved that $0<T^{*}<\infty$ and $\lim\limits_{t\nearrow T^{*}}M(t)=\infty$. In other words, the maximal existence time $T^{*}$ is just the blow-up time of $u$. Moreover, if $\min\limits_{x\in\ol{\O}}u_{0}(x)>0$, then an explicit formula for an upper bound of $T^{*}$ is obtained as below.
\be\label{upper bdd}
T^{*}\leq \frac{1}{(q-1)|\Gamma_1|}\int_{\O}u_0^{1-q}(x)\,dx. \ee
On the other hand, a lower bound of $T^{*}$ is also provided:
\be\label{lower bdd, small broken part}
T^{*}\geq C^{-\frac{2}{N+2}}\bigg[\ln\Big(|\Gamma_1|^{-1}\Big)-(N+2)(q-1)\ln M_0-\ln(q-1)-\ln C\bigg]^{\frac{2}{N+2}}, \ee
where $C=C(N,\O,q)$ is some positive constant. 

In some realistic problems, small $|\Gamma_{1}|$ is of interest. For example in \cite{YZ16}, the motivation for the study of (\ref{Prob}) is the Columbia space shuttle disaster and $\Gamma_{1}$ stands for the broken part on the left wing of the shuttle during launching, so the surface area $|\Gamma_{1}|$ is expected to be small. As $|\Gamma_{1}|\rightarrow 0^{+}$, the upper bound (\ref{upper bdd}) is of order $|\Gamma_{1}|^{-1}$ while the lower bound (\ref{lower bdd, small broken part}) is only of order $\big[\ln\big(|\Gamma_1|^{-1}\big)\big]^{2/(N+2)}$, so there is a big gap between them. The natural question to ask is whether $T^{*}$ grows like a positive power of $|\Gamma_{1}|^{-1}$ or just like the logarithm of $|\Gamma_{1}|^{-1}$ as $|\Gamma_{1}|\rightarrow 0^{+}$? 

Under the convexity assumption of the domain $\O$, this paper provides a lower bound for $T^{*}$ which grows like $|\Gamma_{1}|^{-\alpha}$ for any $\alpha<\frac{1}{N-1}$. In the meantime, the asymptotic behaviour of $T^{*}$ with respect to $q$ (as $q\rightarrow 1^{+}$) and $M_{0}$ (as $M_{0}\rightarrow 0^{+}$ or $M_{0}\rightarrow +\infty$) are also studied. The following is the main result of this paper.
\begin{theorem}\label{Thm, new lower bound of blow-up time}
Assume (\ref{assumption on prob}). Let $\O$ be convex. Then for any $\alpha\in\big[0,\frac{1}{N-1}\big)$, there exists $C=C(N,\O, \alpha)$ such that 
\be\label{new lower bound}
T^{*} \geq \frac{C}{(q-1)M_{0}^{q-1}\,|\Gamma_{1}|^{\alpha}}\,\bigg(\min\bigg\{1,\frac{1}{q M_{0}^{q-1}|\Gamma_{1}|^{\alpha}}\bigg\}\bigg)^{\frac{1+(N-1)\alpha}{1-(N-1)\alpha}}, \ee
where $T^{*}$ is the maximal existence time for (\ref{Prob}) and $M_{0}$ is given by (\ref{initial max}). 
In particular, if $\alpha$ is chosen to be 0 in (\ref{new lower bound}), then 
\be\label{lower bdd for whole bdry}
T^{*}\geq \frac{C_{1}}{(q-1)M_{0}^{q-1}}\min\bigg\{1,\frac{1}{q M_{0}^{q-1}}\bigg\},\ee
for some $C_{1}=C_{1}(N,\O)$.
\end{theorem}

\begin{remark}\label{Remark, asymp behavior}
Theorem \ref{Thm, new lower bound of blow-up time} can be used to study the asymptotic behaviour of $T^{*}$ with respect to $|\Gamma_{1}|$, $q$ and $M_{0}$. See the discussions below.

\begin{itemize}
\item[(1)] Relation between $T^{*}$ and $|\Gamma_{1}|$: If $|\Gamma_{1}|\rightarrow 0^{+}$ and other factors are fixed, then it follows from (\ref{new lower bound}) that for any $\alpha\in\big[0,\frac{1}{N-1}\big)$, 
\[T^{*}\geq \frac{C}{(q-1)M_{0}^{q-1}|\Gamma_{1}|^{\alpha}}\sim |\Gamma_{1}|^{-\alpha}.\]

This lower bound improves the previous one $\big[\ln\big(|\Gamma_1|^{-1}\big)\big]^{2/(N+2)}$ significantly as $|\Gamma_{1}|\rightarrow 0^{+}$. In particular, when the dimension $N=2$, this result is almost optimal since $\alpha$ can be arbitrarily close to 1 and the upper bound in (\ref{upper bdd}) is of order $|\Gamma_1|^{-1}$ as $|\Gamma_{1}|\rightarrow 0^{+}$.

\item[(2)] Relation between $T^{*}$ and $q$: If $q\rightarrow 1^{+}$ and other factors are fixed, then (\ref{lower bdd for whole bdry}) implies 
\[T^{*}\geq \frac{C_{1}}{q-1}.\]
On the other hand, 
the upper bound (\ref{upper bdd}) implies 
$$T^{*}\leq \frac{C_{2}}{q-1}.$$ 
Thus the order of $T^{*}$ as $q\rightarrow 1^{+}$ is exactly $(q-1)^{-1}$.

\item[(3)] Relation between $T^{*}$ and $M_{0}$:
\begin{itemize}
\item[$\bullet$] If $M_{0}\rightarrow 0^{+}$ and other factors are fixed, then (\ref{lower bdd for whole bdry}) implies 
$$T^{*}\geq C_{1}\,M_{0}^{-(q-1)}.$$
This order is optimal since if the initial data $u_{0}$ is kept to be constant, then it follows from (\ref{upper bdd}) that 
\[T^{*}\leq C_{2}\,M_{0}^{-(q-1)}.\]

\item[$\bullet$] If $M_{0}\rightarrow +\infty$ and other factors are fixed, then (\ref{lower bdd for whole bdry}) implies 
$$T^{*}\geq C_{1}\,M_{0}^{-2(q-1)}.$$
In this case, there is gap from the upper bound in (\ref{upper bdd}).
\end{itemize} 
\end{itemize}
\end{remark}

\subsection{Historical Works}
\subsubsection{Blow-up phenomenon for the heat equation with nonlinear Neumann conditions}
Starting from the pioneering papers by Kaplan \cite{Kap63} and Fujita\cite{Fuj66}, the blow-up phenomenon of parabolic type has been extensive studied in the literature for the Cauchy problem as well as the boundary value problems. We refer the readers to the surveys \cite{DL00, Lev90}, the books \cite{Hu11, QS07} and the references therein. 

%In the seminal work \cite{Fuj66}, Fujita studied the Cauchy problem
%\[\left\{\begin{array}{lcl}
%u_{t}(x,t)=\Delta u(x,t)+u^{p}(x,t) & \text{in} & \m{R}^{N}\times(0,\infty),\\
%u(x,0)=\psi(x) &\text{in} & \m{R}^{N},
%\end{array}\right.\]
%where $\psi\in C^{2}(\m{R}^{N})$ is nonnegative and $\psi$, $D_{i}\psi$, $D_{ij}\psi$ are all bounded on $\m{R}^{N}$. It is shown that if $1<p<1+\frac{2}{N}$, then the only nonnegative global solution is $u\equiv 0$, and if $p>1+\frac{2}{N}$, then there exists a positive global solution when $\psi$ is positive and sufficiently small. 

%The research topics include the existence and uniqueness of the solution, (e.g. \cite{Ama86a, Ama86b, ACR-B99, CDW09, LN92, L-GMW91, Wal75}); 
%finite-time blowup of the solution and the upper bound estimate for the blow-up time (e.g. \cite{HY94, LN92, Lev73, Lev75, LP74, L-GMW91, PPV-P10, RR97, Wal75}); 
%lower bound estimate for the blow-up time (e.g. \cite{LL12, PPV-P10, PS06, PS07, PS09});
%blow-up sets, blow-up rate and the asymptotic behaviour of the solutions near the blow-up time (e.g. \cite{FM85, GK85, HY94, LN92, L-GMW91, MW85, RR97, Wei85}).

One of the core objects in the area is the heat equation with Neumann boundary conditions in a bounded domain $\O$:
\be\label{heat with Neumann}
\left\{\begin{array}{lll}
u_{t}(x,t)=\Delta u(x,t) &\text{in}& \Omega\times (0,T],\\
\frac{\partial u}{\partial n}(x,t)=F\big(u(x,t)\big) &\text{on}& \p\O\times (0,T],\\
u(x,0)=\psi(x) &\text{in}& \Omega.
\end{array}\right.\ee
Here the initial data $\psi$ is not assumed to be nonnegative. It is well-known that there are two ways to construct the classical solution to (\ref{heat with Neumann}) depending on the smoothness of $\p\O$, $F$ and $u_{0}$ (see Theorem 1.1 and 1.3 in \cite{L-GMW91}, also see the books \cite{Fri64, Lie96, LSU68}).
\begin{itemize}
\item[(a)] The first way is by Schauder estimate. Assume $\p\O$ is $C^{2+\alpha}$, $F\in C^{1+\alpha}(\m{R})$, $\psi\in C^{2+\alpha}(\ol{\O})$ and the compatibility condition
\[\frac{\p \psi}{\p n}(x)=F\big(\psi(x)\big), \quad \forall\, x\in \p\O.\]

Then there exists $T>0$ and a unique function $u$ in $C^{2+\alpha,1+\frac{\alpha}{2}}\big(\ol{\O}\times[0,T]\big)$ which satisfies (\ref{heat with Neumann}) pointwisely.

\item[(b)] The second way is by the heat potential technique. The requirements on the data can be relaxed and in particular, the compatibility condition is no longer needed, but accordingly the conclusion is also weaker. Assume $\p\O$ is $C^{1+\alpha}$, $F\in C^{1}(\m{R})$, $\psi\in C^{1}(\ol{\O})$. Then there exists $T>0$ and a unique function $u$ in $C^{2,1}\big(\O\times(0,T]\big)\bigcap C\big(\overline{\O}\times[0,T]\big)$ which satisfies (\ref{heat with Neumann}) pointwisely.
\end{itemize}
In most papers, the assumptions will fall into either case (a) or case (b) for which the meaning of the finite-time blowup is clear. For other cases, the finite-time blowup means by assuming the local existence and uniqueness of a classical solution, then such a local solution can not be extended to a global one. In the following statements, we will ignore these distinctions and just refer them to be the local (classical) solutions. It has been already known that if $F$ is bounded on $\m{R}$, then the local solutions can be extended globally. But if $F$ is unbounded, then the finite-time blowup may occur. 

The first result on the blow-up phenomenon for (\ref{heat with Neumann}) is due to Levine and Payne \cite{LP74}. They used a concavity argument to conclude that any classical solution blows up in finite time under the two assumptions as below.
\begin{itemize}
\item First, 
\be\label{LP radiation fn}
F(z)=|z|^{q}h(z),\ee
for some constant $q>1$ and some differentiable, increasing function $h(z)$.

\item Secondly,  
\be\label{LP cond}
\frac{1}{|\p\O|}\,\int_{\p\O}\bigg(\int_{0}^{\psi(x)}F(z)\,dz\bigg)\,dS(x)>\frac{1}{2}\int_{\O}|D\psi(x)|^{2}\,dx. \ee
\end{itemize}
\begin{remark}\label{Remark, cor of LP74} As a corollary of the result in \cite{LP74}, if $h(z)$ is positive and $\psi$ is a positive constant function, then the solution blows up in finite time. Combining with the maximum principle, this also implies that for any positive $h(z)$ and $\psi$, the solution blows up in finite time. \end{remark} 
Later, Walter \cite{Wal75} gave a more complete characterization for the blow-up phenomenon by introducing some comparison functions. More precisely, if $F(z)$ is positive, increasing and convex for $z\geq z_{0}$ with some constant $z_{0}$, there are exactly two possibilities.
\begin{itemize}
\item First, if $\int_{z_{0}}^{\infty}\frac{1}{F(z)F'(z)}\,dz=\infty$, then the solution exists globally for any initial data $\psi$.

\item Secondly, if $\int_{z_{0}}^{\infty}\frac{1}{F(z)F'(z)}\,dz<\infty$, then the solution blows up in finite time for large initial data $\psi$.
\end{itemize}
The result was further generalized by Rial and Rossi \cite{RR97} (also see \cite{L-GMW91}). In \cite{RR97}, by assuming $F$ to be $C^{2}$, increasing and positive in $\m{R}_{+}$, and also assuming $1/F$ to be locally integrable near $\infty$ (that is $\int^{\infty}\frac{1}{F(z)}\,dz<\infty$), it is shown that for any positive initial data $\psi>0$, the classical solution blows up in finite time. The success of their method was due to a clever choice of an energy function which made the proof short and elementary.

Applying these earlier results to the simpler model (that is (\ref{Prob}) with $\Gamma_{2}=\emptyset$)
\be\label{heat with Neumann, power}
\left\{\begin{array}{lll}
u_{t}(x,t)=\Delta u(x,t) &\text{in}& \Omega\times (0,T],\\
\frac{\partial u}{\partial n}(x,t)=u^{q}(x,t) &\text{on}& \p\O\times (0,T],\\
u(x,0)=u_{0}(x) &\text{in}& \Omega,
\end{array}\right.\ee
where $q>1$ and the initial data $u_{0}\geq 0$ and $u_{0}\not\equiv 0$, it can be shown that any solution to (\ref{heat with Neumann, power}) blows up in finite time. In fact, by the maximum principle, the solution $u$ becomes positive as soon as $t>0$. Then either Remark \ref{Remark, cor of LP74} or the result in \cite{RR97} (also see \cite{HY94}) implies the finite time blowup of the solution. However, when the nonlinear radiation condition is only imposed on partial boundary (that is when $\Gamma_{2}\neq\emptyset$ in (\ref{Prob})), additional difficulties appear due to the discontinuity of the normal derivative along the interface $\wt{\Gamma}$ between $\Gamma_{1}$ and $\Gamma_{2}$. To our knowledge, \cite{YZ16} was the first paper that dealt with this problem and gave the bounds of the blow-up time as in (\ref{upper bdd}) and (\ref{lower bdd, small broken part}).

\subsubsection{Estimate for the lower bound of the blow-up time}
When considering the bounds of the blow-up time, the upper bound is usually related to the nonexistence of the global solutions and various methods on this issue have been developed (see \cite{Lev75} for a list of six methods). The lower bound was not studied as much in the past and not many methods have been explored. However, the lower bound can be argued to be more useful in practice, since it provides an estimate of the safe time. In the existing literature, they either used a comparison argument or applied differential inequality techniques. 

The first work on the lower bound estimate of the blow-up time was due to Kaplan \cite{Kap63}. Later, Payne and Schaefer  developed a very robust method on this issue. For example, they derived the lower bound of the blow-up time for the nonlinear heat equation with homogeneous Dirichlet or Neumann boundary conditions in \cite{PS06, PS07}. Later this idea was also applied to the problem (\ref{heat with Neumann}) (see \cite{PS09}) and many other types of problems (see e.g. \cite{Ena11, PPV-P10, PP13, BS14, LL12, TV-P16, AD17, DS16}). Recently, in order to obtain the lower bound of blow-up time for the problem (\ref{Prob}), the authors of this paper developed another method in \cite{YZ16} by analyzing the representation formula of the solution. Nonetheless, this method still relied on the differential inequality argument. In the next section, we will compare this current paper with \cite{PS09} and \cite{YZ16} in more detail.

\subsection{Comparison}
In this section, we will compare the methods and results in \cite{PS09} and \cite{YZ16} with those in the current paper concerning the model problem (\ref{Prob}).

\subsubsection{Methods}
In \cite{PS09}, due to technical reasons, it required $\O$ to be convex in $\m{R}^{3}$, $\Gamma_{1}=\p\O$ and $q\geq 3/2$.  By introducing the energy function
\bes \varphi(t)=\int_{\O}u^{4(q-1)}(x,t)\,dx\ees
and adopting a Sobolev-type inequality developed in \cite{PS06}, they derived a first order differential inequality for $\varphi(t)$ and then obtain a lower bound for $T^{*}$:
\be\label{PS result}T^{*}\geq C\int_{\varphi(0)}^{\infty}\frac{d\eta}{\eta+\eta^{3}},\ee
where $C$ is some positive constant depending only on $\O$ and $q$. 

\cite{YZ16} studied the problem (\ref{Prob}) without any further assumptions. The authors considered the function $M(t)$ defined in (\ref{max function at time t}). By analysing the representation formula (\ref{rep formula}) for the solution, they derive a nonlinear Gronwall-type inequality for $M^{N+2}(t)$ and then obtain the lower bound (\ref{lower bdd, small broken part}).

For this paper, also due to the technical reasons, the convexity of $\O$ is imposed, but this is the only additional assumption. The novelty of this paper is that we develop a new method which does not introduce any functionals or differential inequalities. Instead, we will chop the maximal existence time interval $[0,T^{*})$ into small subintervals in a delicate way such that the length of each subinterval is at least $t_{*}>0$ and the total number of the subintervals has a lower bound $L_{0}$. As a result, $L_{0} t_{*}$ is a lower bound for $T^{*}$ as stated in Theorem \ref{Thm, new lower bound of blow-up time}. 

\subsubsection{Results}
Finally we want to compare Theorem \ref{Thm, new lower bound of blow-up time} with the previous lower bound estimates (\ref{lower bdd, small broken part}) in \cite{YZ16} and (\ref{PS result}) in \cite{PS09}. For convenience of statement, the lower bounds in (\ref{lower bdd, small broken part}), (\ref{PS result}), (\ref{new lower bound}) and (\ref{lower bdd for whole bdry}) will be written as $T_{1}$, $T_{2}$, $T_{1*}$ and $T_{2*}$ respectively. 

\begin{itemize}
\item Comparing Theorem \ref{Thm, new lower bound of blow-up time} with (\ref{lower bdd, small broken part}): 
\begin{itemize}
\item[$\diamond$] As $|\Gamma_{1}|\rightarrow 0^{+}$, this has been discussed before: $T_{1*}$ improves $T_{1}$ significantly.

\item[$\diamond$] As $q\rightarrow 1^{+}$, $T_{2*}$ is of order $(q-1)^{-1}$; however, $T_{1}$ is only of a logarithm order $\big(\ln[(q-1)^{-1}]\big)^{2/(N+2)}$.

\item[$\diamond$] As $M_{0}\rightarrow 0^{+}$, $T_{2*}$ is of order $M_{0}^{-(q-1)}$; on the other hand, $T_{1}$ is only of a logarithm order $\big(\ln[M_{0}^{-1}]\big)^{2/(N+2)}$.

\item[$\diamond$] As $M_{0}\rightarrow \infty$, $T_{2*}$ is of order $M_{0}^{-2(q-1)}$; but $T_{1}$ is not applicable since it is negative when $M_{0}$ is large.
\end{itemize}

\item Comparing Theorem \ref{Thm, new lower bound of blow-up time} with (\ref{PS result}). 
\begin{itemize}
\item[$\diamond$] On the one hand, $T_{2*}$ is valid for any $q>1$ and also give the exact asymptotic rate of $T^{*}$ as $q\rightarrow 1^{+}$. On the other hand, \cite{PS09} requires $q\geq 3/2$ due to the technical restriction.

\item[$\diamond$] $T_{2*}$ holds for any nonempty partial boundary $\Gamma_{1}$, but \cite{PS09} only considers the whole boundary case $\Gamma_{1}=\p\O$.

\item[$\diamond$] As $M_{0}\rightarrow 0^{+}$, $T_{2*}$ is of order $M_{0}^{-(q-1)}$; if the initial data $u_{0}$ does not oscillate too much, that is assuming
\[\varphi(0)=\int_{\O}u_{0}^{4(q-1)}(x)\,dx\sim M_{0}^{4(q-1)}\,|\O|,\]
then \[T_{2}\sim\ln\Big(\frac{1}{\varphi(0)}\Big)\sim 4(q-1)\ln(M_{0}^{-1}),\]
which is only a logarithm order $\ln(M_{0}^{-1})$.

\item[$\diamond$] As $M_{0}\rightarrow \infty$, $T_{2*}$ is of order $M_{0}^{-2(q-1)}$; again, if the initial data $u_{0}$ again does not oscillate too much, then 
\[T_{2}\sim [\varphi(0)]^{-2}\sim M_{0}^{-8(q-1)}.\]
Since $M_{0}$ is large, 
\[M_{0}^{-2(q-1)}>>M_{0}^{-8(q-1)}.\]
\end{itemize}

\end{itemize}

\subsection{Future Works}
\begin{itemize}
\item In practice, the domain $\O$ may not be convex everywhere. Only local convexity is reasonable. In the subsequent work \cite{YZ1707}, we follow the strategy in this paper and obtain similar results by only assuming local convexity near $\Gamma_{1}$. But this method fails when the region near $\Gamma_{1}$ is concave. So how to deal with this case is an open problem.

\item Even in the convex domain case, there is still gap between the lower bound $|\Gamma_{1}|^{-1/(N-1)}$ and the upper bound $|\Gamma_1|^{-1}$ as $|\Gamma_{1}|\rightarrow 0^{+}$. It is an interesting question whether we can further narrow this gap.
\end{itemize}

\subsection{Organization}
The organization of this paper is as follows. In Section \ref{Sec, preliminary}, we state several basic results which will be used later. Section \ref{Sec, lower bdd in convex domain} is devoted to demonstrate the main idea and a detailed proof of Theorem \ref{Thm, new lower bound of blow-up time}. The appendix justifies a representation formula for the solution which plays an essential role in the proof of the main theorem.

\section{Preliminary results}\label{Sec, preliminary}
The first result is about the continuity in time $t$ (as $t\rightarrow 0$) of the integral of the fundamental solution to the heat equation over the domain $\O$. 

\begin{lemma}\label{Lemma, lower bdd, int of fund soln in finite time}
Let $\O$ be a bounded open subset in $\m{R}^{N}(N\geq 2)$ with $C^{2}$ boundary. Define $F:\p\O\times[0,1]\rightarrow\m{R}$ by
\begin{align*}
F(x,t)=\left\{\begin{array}{lll}
\int_{\O}\Phi(x-y,t)\,dy &\text{for}& x\in\p\O,\, t\in(0,1],\\
1/2 &\text{for}& x\in\p\O,\, t=0.
\end{array}\right.\end{align*}
Then $F$ is continuous on $\p\O\times[0,1]$. As a result,
\be\label{lower bdd, int of fund soln in finite time}
b_1\triangleq \min\limits_{\p\O\times[0,1]} F\ee
is a positive constant depending only on $\O$ and the dimension $N$.
\end{lemma}
\begin{proof}
Since $\p\O$ has been assumed to be $C^{2}$, the proof can be carried out by standard analysis. We can also prove it by applying (\ref{identity, bdry}) and noticing the uniform decay of the term $$\int_{0}^{t}\int_{\p\O}\Big| D_{y}\big[\Phi(x-y,t-\tau)\big]\cdot\ora{n}(y)\Big|\,dS(y)\,d\tau$$ in $x\in\p\O$ as $t\rightarrow 0^{+}$. The details are omitted here.
\end{proof}

We want to remark that although $\int_{\m{R}^{N}}\Phi(x-y,t)\,dy=1$, the main contribution of the integral, when $t$ is small, only comes from the integral over a small ball $B_{r}(x)$ around $x$ (the smaller $t$ is, the smaller $r$ can be chosen). Now if the integral domain is a bounded region $\O$ instead of $\m{R}^{N}$ and if $x\in\p\O$, then the intersection $B_{r}(x)\bigcap \O$ will be nearly half of $B_{r}(x)$ as $r\rightarrow 0^{+}$. As a result, when $t\rightarrow 0^{+}$, the limit of $\int_{\O}\Phi(x-y,t)\,dy$ becomes 1/2 instead of 1.

The second result shows two identities concerning the fundamental solution to the heat equation.

\begin{lemma}\label{Lemma, identity}
Let $\O$ be a bounded open subset in $\m{R}^{N}(N\geq 2)$ with $C^{2}$ boundary. Define $\Phi$ as in (\ref{fund soln of heat eq}). Then 
\be\label{identity, bdry}
\int_{\O}\Phi(x-y,t)\,dy-\int_{0}^{t}\int_{\p\O}D_{y}\big[\Phi(x-y,t-\tau)\big]\cdot\ora{n}(y)\,dS(y)\,d\tau=\frac{1}{2}, \quad\forall\, x\in\p\O,\,t>0.\ee
In addition, if $\O$ is convex, then 
\be\label{identity, convex bdry}
\int_{\O}\Phi(x-y,t)\,dy+\int_{0}^{t}\int_{\p\O}\Big| D_{y}\big[\Phi(x-y,t-\tau)\big]\cdot\ora{n}(y)\Big|\,dS(y)\,d\tau=\frac{1}{2}, \quad\forall\, x\in\p\O,\,t>0. \ee
\end{lemma}

\begin{proof}
Consider the problem 
\begin{equation}\label{test model}
\left\{\begin{array}{lll}
u_{t}(x,t)=\Delta u(x,t) &\text{in}& \Omega\times (0,\infty),\\
\frac{\partial u}{\partial n}(x,t)=0 &\text{on}& \p\O\times (0,\infty),\\
u(x,0)=1 &\text{in}& \Omega,
\end{array}\right.
\end{equation}
it obviously has the unique solution $u\equiv 1$ on $\ol{\O}\times[0,\infty)$. As a result, (\ref{identity, bdry}) follows by plugging $u\equiv 1$ into the representation formula (\ref{rep formula}) (taking $\Gamma_{1}=\emptyset$). Now if $\O$ is convex, then $D_{y}\big[\Phi(x-y,t-\tau)\big]\cdot\ora{n}(y)\leq 0$ for any $x,\,y\in\p\O$ and $0\leq \tau<t$. Thus, (\ref{identity, bdry}) implies (\ref{identity, convex bdry}).
\end{proof}

Finally, we present an elementary boundary-time integral estimate which is obtained by applying Holder's inequality. For the convenience of notation, for any $\alpha\in\big[0,\frac{1}{N-1}\big)$, we denote
\be\label{def of N_(alpha)}
N_{\alpha}=\frac{1-(N-1)\alpha}{2}.\ee
It is readily seen that $0<N_{\alpha}\leq \frac{1}{2}$.
\begin{lemma}\label{Lemma, bound by power of Gamma_1}
Let $\O$ be a bounded open subset in $\m{R}^{N}(N\geq 2)$ with $C^{2}$ boundary. Then there exists $C=C(N,\O)$ such that for any $\Gamma_{1}\subseteq \O$, $\alpha\in\big[0,\frac{1}{N-1}\big)$, $x\in\p\O$ and $t>0$,
\be\label{bound by power of gamma 1}
\int_{0}^{t}\int_{\Gamma_1}\Phi(x-y,t-\tau)\,dS(y)\,d\tau\leq \frac{C\,|\Gamma_1|^{\alpha}\,t^{N_{\alpha}}}{N_{\alpha}},\ee
where $N_{\alpha}$ is defined as in (\ref{def of N_(alpha)}). In particular, if $\alpha=0$, then
\[\int_{0}^{t}\int_{\Gamma_1}\Phi(x-y,t-\tau)\,dS(y)\,d\tau\leq C\,\sqrt{t}.\]
\end{lemma}
\begin{proof}
Fix $\Gamma_{1}\subseteq\O$, $\alpha\in\big[0,\frac{1}{N-1}\big)$, $x\in\p\O$ and $t>0$. We denote 
\[LHS=\int_{0}^{t}\int_{\Gamma_1}\Phi(x-y,t-\tau)\,dS(y)\,d\tau.\]
By a change of variable in time, 
\begin{align}
LHS &=\int_{0}^{t}\int_{\Gamma_1}\Phi(x-y,\tau)\,dS(y)\,d\tau \notag\\
&=\frac{1}{(4\pi)^{N/2}}\int_{0}^{t}\tau^{-N/2}\int_{\Gamma_1}e^{-|x-y|^{2}/(4\tau)}\,dS(y)\,d\tau. \label{int over gamma 1 and time}
\end{align}
For any $m\geq 1$, applying Holder's inequality,
\be\label{int over gamma 1}
\int_{\Gamma_1}e^{-|x-y|^{2}/(4\tau)}\,dS(y)\leq \bigg(\int_{\Gamma_1}e^{-m\,|x-y|^{2}/(4\tau)}\,dS(y)\bigg)^{1/m}\,|\Gamma_1|^{(m-1)/m}. \ee
Denote 
\be\label{upper bdd, int of fund soln over bdry}
B_1=\sup_{\tau>0}\,\sup_{x\in\p\O}\,\tau^{-\frac{N-1}{2}}\int_{\p\O}e^{-|x-y|^2/(4\tau)}\,dS(y).\ee
It is shown in Lemma 3.1 of \cite{YZ16} that $B_1$ is a finite positive constant depending only on $\O$ and $N$. As a result, 
\begin{align*}
\int_{\Gamma_1}e^{-m\,|x-y|^{2}/(4\tau)}\,dS(y) &= \int_{\Gamma_1} e^{-\,|x-y|^{2}\big/\big[4(\tau/m)\big]}\,dS(y) \\
&\leq \Big(\frac{\tau}{m}\Big)^{(N-1)/2}\,B_1 \\
& \leq \tau^{(N-1)/2}\,B_1.
\end{align*}
Combining this inequality with (\ref{int over gamma 1}), 
\begin{align}
\int_{\Gamma_1}e^{-|x-y|^{2}/(4\tau)}\,dS(y) &\leq B_{1}^{1/m}\,\tau^{(N-1)/(2m)}\,|\Gamma_1|^{(m-1)/m} \notag\\
&\leq (B_{1}+1)\,\tau^{(N-1)/(2m)}\,|\Gamma_1|^{(m-1)/m}. \label{bound for int over bdry in m}
\end{align} 
Plugging (\ref{bound for int over bdry in m}) into (\ref{int over gamma 1 and time}),
\be\label{LHS bound by m}
LHS\leq \frac{B_{1}+1}{(4\pi)^{N/2}}\,|\Gamma_1|^{(m-1)/m}\int_{0}^{t}\tau^{-\frac{N}{2}+\frac{N-1}{2m}}\,d\tau.\ee
Choose $$m=\dfrac{1}{1-\alpha}.$$
Then $m\geq 1$ and $(m-1)/m=\alpha$. Therefore, (\ref{LHS bound by m}) becomes
\begin{align*}
LHS &\leq \frac{B_{1}+1}{(4\pi)^{N/2}}\,|\Gamma_1|^{\alpha}\int_{0}^{t}\tau^{\frac{1-(N-1)\alpha}{2}-1}\,d\tau\\
&= \frac{2\,(B_1+1)}{(4\pi)^{N/2}\,\big[1-(N-1)\alpha\big]}\,|\Gamma_1|^{\alpha}\,t^{\frac{1-(N-1)\alpha}{2}},
\end{align*}
where the last equality takes advantage of the assumption that $\alpha<\frac{1}{N-1}$.
\end{proof}

\section{Lower bounds in the convex domain case}
\label{Sec, lower bdd in convex domain}
\subsection{Main idea}
In this section, $\O$ is assumed to be convex. In addition, $M_{0}$ and $M(t)$ are still defined as in (\ref{initial max}) and (\ref{max function at time t}). First of all, let us recall the method in \cite{YZ16} on the lower bound estimate of $T^{*}$. By considering the maximum of $u$ on $\p\O$ for each time $t$, the authors introduce
\be\label{bdry max}\wt{M}(t)= \max_{x\in\p\O}u(x,t).\ee
Note $\wt{M}(t)$ is different from $M(t)$. For any $t>0$, there exists $x^{0}\in\p\O$ such that $u(x^{0},t)=\wt{M}(t)$, so it follows from the representation formula (\ref{rep formula}) that
\begin{align}
\wt{M}(t) \leq &\,\, 2M_{0}\int_{\O}\Phi(x^{0}-y,t)\,dy+2\int_{0}^{t}\wt{M}(\tau)\int_{\p\O}\Big| D_{y}\big[\Phi(x^{0}-y,t-\tau)\big]\cdot\ora{n}(y)\Big|\,dS(y)\,d\tau \notag\\
&+2\int_{0}^{t}\wt{M}^{q}(\tau)\int_{\Gamma_1}\Phi(x^{0}-y,t-\tau)\,dS(y)\,d\tau \notag\\
\triangleq &\,\, I+II+III.
\label{index for terms}
\end{align}
After estimating
\[\int_{\p\O}\Big| D_{y}\big[\Phi(x^{0}-y,t-\tau)\big]\cdot\ora{n}(y)\Big|\,dS(y) \quad\text{and}\quad \int_{\Gamma_1}\Phi(x^{0}-y,t-\tau)\,dS(y),\]
in the terms $II$ and $III$, they achieve the lower bound (\ref{lower bdd, small broken part}) by applying a Gronwall-type technique to (\ref{index for terms}). However, this lower bound is only logarithm of $|\Gamma_{1}|^{-1}$ as $|\Gamma_{1}|\rightarrow 0^{+}$. The obstruction that prevents this method obtaining a polynomial order of $|\Gamma_{1}|^{-1}$ is explained through the remark below.

\begin{remark}
Consider the following two simple integral inequalities. 
First,
\be\label{first ex}\left\{\begin{array}{l}
\phi_{1}(t)\leq A+\int_{0}^{t}\phi_{1}(\tau)\,d\tau+|\Gamma_{1}|\int_{0}^{t}\phi_{1}^{q}(\tau)\,d\tau, \quad t>0, \\
\phi_{1}(0)=A>0.
\end{array}\right.\ee
It is easy to see by the Gronwall's inequality that the blow-up time $T_{1}^{*}$ of (\ref{first ex}) satisfies
\[T_{1}^{*}\geq \frac{1}{q-1}\ln\bigg(1+\frac{1}{A^{q-1}\,|\Gamma_{1}|}\bigg),\]
which is of order $\ln(|\Gamma_{1}|^{-1})$ as $|\Gamma_{1}|\rightarrow 0^{+}$.
Secondly,
\be\label{second ex}\left\{\begin{array}{l}
\phi_{2}(t)\leq A+|\Gamma_{1}|\int_{0}^{t}\phi_{2}^{q}(\tau)\,d\tau, \quad t>0, \\
\phi_{2}(0)=A>0.
\end{array}\right.\ee
Again by applying Gronwall's inequality, the blow-up time $T_{2}^{*}$ of (\ref{second ex}) satisfies
\[T_{2}^{*}\geq \frac{1}{(q-1)A^{q-1}|\Gamma_{1}|},\]
which is of order $|\Gamma_{1}|^{-1}$. Comparing these two differential equations, (\ref{first ex}) contains a linear term $\int_{0}^{t}\phi_{1}(\tau)\,d\tau$, however, (\ref{second ex}) does not. So this term is the obstruction that prevents the lower bound being a polynomial order of $|\Gamma_{1}|^{-1}$. 
\end{remark}

Thus, coming back to (\ref{index for terms}), if the linear term II can be eliminated, then the lower bound is expected to be a polynomial order of $|\Gamma_{1}|^{-1}$. Under the convexity assumption of $\O$, we will develop a new method which  eliminates the linear term II in a certain sense \big(see (\ref{max ineq, kth step, simplify}) through (\ref{max ineq, kth step, interm})\big). In addition, this new method is much more accurate than the Gronwall's type argument. The details will be shown in the proof of Theorem \ref{Thm, new lower bound of blow-up time}. 

The initial idea of the proof is as follows. First, we chop the range of $M(t)$ into small pieces $[M_{k-1},M_{k}]$ $(k\geq 1)$ and denote $t_{k}$ to be the time that $M(t)$  increases from $M_{k-1}$ to $M_{k}$. Secondly, we will find a lower bound $t_{k*}$ for each $t_{k}$. Finally, $\sum_{k=0}^{\infty}t_{k*}$ is a lower bound for $T^{*}$. However, it is very difficult to find a precise relation between the sum $\sum_{k=1}^{\infty}t_{k*}$ and the order of $|\Gamma_{1}|^{-1}$ since the expression of $\sum_{k=1}^{\infty}t_{k*}$ is very complicated. As a result, in order to obtain a lower bound with simple formula, what we actually do in the proof of Theorem \ref{Thm, new lower bound of blow-up time} is to find a uniform lower bound $t_{*}$ for finitely many $t_{k}$, say $1\leq k\leq L$, then $Lt_{*}$ serves as a lower bound. This way greatly reduces the computations and yields a lower bound with nice expression. The analysis based on the representation formula (\ref{rep formula from any time}) and the identity (\ref{identity, convex bdry}). The delicate part is how to choose $\{M_{k}\}_{1\leq k\leq L}$ and $t_{*}$ such that $Lt_{*}$ is maximized. 

%It is also interesting to follow the original idea and estimate $\sum_{k=1}^{\infty}t_{k*}$ directly. It may be possible to obtain a larger lower bound. But this paper does not pursuit this direction.

\subsection{Proof of Theorem \ref{Thm, new lower bound of blow-up time}}
Now we start to prove the main result of this paper.
\begin{proof}[{\bf Proof of Theorem \ref{Thm, new lower bound of blow-up time}.}]

Let $M(t)$ be defined as in (\ref{max function at time t}) and let $t_{*}\in(0, 1]$ be a positive constant which will be determined later. We will first use induction to construct a finite strictly increasing sequence $\{M_{j}\}_{0\leq j\leq L}$ such that if $T_{k}$ denotes the first time that $M(t)$ reaches $M_{k}$, then $T_{k}\geq kt_{*}$ for $0\leq k\leq L$ (the total steps $L$ depends on $t_{*}$). Then $t_{*}$ will be chosen so that $Lt_{*}$ is maximized.

\begin{itemize}
\item {\bf Step $\bf 0$.} Define $M_{0}$ as in (\ref{initial max}). Then $T_{0}=0$.

\item {\bf Step $\bf k$ $\bf (k\geq 1)$.} Suppose $\{M_{j}\}_{0\leq j\leq k-1}$ have been defined such that $T_{k-1}\geq (k-1)t_{*}$. We intend to construct $M_{k}$ in this step such that $T_{k}\geq k t_{*}$. For some $\lam_{k}>1$ to be determined, we define 
\be\label{notation in kth step}
M_{k}=\lam_{k}M_{k-1} \ee
and denote 
\be\label{difference in time, kth step}
t_{k}=T_{k}-T_{k-1}. \ee
Next we want to choose suitable $\lambda_{k}$ such that $t_{k}\geq t_{*}$, which implies $T_{k}\geq kt_{*}$. 

Since $M_{k}>M_{0}$ and $T_{k}$ is the first time that $M(t)$ reaches $M_{k}$, it follows from the maximum principle that   there exists $x^{k}\in\p\O$ such that $u(x^{k},T_{k})=M_{k}$. Applying the time-shifted representation formula (\ref{rep formula from any time}) with $T=T_{k-1}$ and $(x,t)=(x^{k}, t_{k})$, 
\begin{align}
u(x^{k},T_{k}) = &\,\, 2\int_{\O}\Phi(x^{k}-y,t_{k})\,u(y,T_{k-1})\,dy \notag\\
&-2\int_{0}^{t_{k}}\int_{\p\O}D_{y}\big[\Phi(x^{k}-y, t_{k}-\tau)\big]\cdot \ora{n}(y)\,u(y,T_{k-1}+\tau)\,dS(y)\,d\tau \notag\\
&+2\int_{0}^{ t_{k}}\int_{\Gamma_1}\Phi(x^{k}-y, t_{k}-\tau)\,u^{q}(y,T_{k-1}+\tau)\,dS(y)\,d\tau \label{rep formula for kth step}.  
\end{align}
As a result, 
\begin{align}
M_{k} \leq &\,\, 2M_{k-1}\int_{\O}\Phi(x^{k}-y, t_{k})\,dy \notag\\
&+2M_{k}\int_{0}^{ t_{k}}\int_{\p\O}\Big|D_{y}\big[\Phi(x^{k}-y,t_{k}-\tau)\big]\cdot \ora{n}(y)\Big|\,dS(y)\,d\tau \notag\\
&+2M_{k}^{q}\int_{0}^{ t_{k}}\int_{\Gamma_1}\Phi(x^{k}-y, t_{k}-\tau)\,dS(y)\,d\tau. \label{max ineq, kth step, simplify}
\end{align}
Since $\O$ is convex, it follows from (\ref{identity, convex bdry}) that
\[\int_{0}^{ t_{k}}\int_{\p\O}\Big| D_{y}\big[\Phi(x^{k}-y, t_{k}-\tau)\big]\cdot\ora{n}(y)\Big|\,dS(y)\,d\tau=\frac{1}{2}-\int_{\O}\Phi(x^{k}-y, t_{k})\,dy.\]
Plugging this identity into (\ref{max ineq, kth step, simplify}) and simplifying, 
\begin{align}
M_{k}\int_{\O}\Phi(x^{k}-y, t_{k})\,dy\leq M_{k-1}\int_{\O}\Phi(x^{k}-y, t_{k})\,dy+M_{k}^{q}\int_{0}^{ t_{k}}\int_{\Gamma_1}\Phi(x^{k}-y, t_{k}-\tau)\,dS(y)\,d\tau \label{max ineq, kth step, interm}.
\end{align} 
Equivalently,
\begin{align}
(M_{k}-M_{k-1})\int_{\O}\Phi(x^{k}-y, t_{k})\,dy\leq M_{k}^{q}\int_{0}^{ t_{k}}\int_{\Gamma_1}\Phi(x^{k}-y, t_{k}-\tau)\,dS(y)\,d\tau \label{max ineq, kth step, final}.
\end{align} 

{\bf If $t_{k}>1$, we automatically have $t_{k}\geq t_{*}$ since $t_{*}\leq 1$. So next we assume $t_{k}\leq 1$.} Then it follows from Lemma \ref{Lemma, lower bdd, int of fund soln in finite time} that
\be\label{lower bdd for LHS, kth step}
\int_{\O}\Phi(x^{k}-y, t_{k})\,dy\geq b_1. \ee
In addition, Lemma \ref{Lemma, bound by power of Gamma_1} implies the existence of a constant $C=C(N,\O)$ such that
\begin{align}
\int_{0}^{ t_{k}}\int_{\Gamma_1}\Phi(x^{k}-y, t_{k}-\tau)\,dS(y)\,d\tau &\leq \frac{C\,|\Gamma_1|^{\alpha}\,t_{k}^{N_{\alpha}}}{N_{\alpha}}, \label{upper bdd for RHS, kth step}
\end{align}
where $N_{\alpha}$ is defined as in (\ref{def of N_(alpha)}). Plugging (\ref{lower bdd for LHS, kth step}) and (\ref{upper bdd for RHS, kth step}) into (\ref{max ineq, kth step, final}), 
\be\label{relation for max value, kth}\frac{M_{k}-M_{k-1}}{M_{k}^{q}}\leq \frac{C\,|\Gamma_1|^{\alpha}\,t_{k}^{N_{\alpha}}}{b_1 N_{\alpha}}. \ee
Recalling that $M_{k}=\lam_{k}M_{k-1}$, then
\be\label{relation of lamda and M for small time, kth}
\frac{\lam_{k}-1}{\lam_{k}^{q}M_{k-1}^{q-1}}\leq \frac{C\,|\Gamma_1|^{\alpha}\,t_{k}^{N_{\alpha}}}{b_1 N_{\alpha}}. \ee

Based on this observation, {\bf if there exists $\lam_{k}>1$ such that }
\be\label{choice of lam, kth}
\frac{\lam_{k}-1}{\lam_{k}^{q}\,M_{k-1}^{q-1}}= \delta_{1}, \ee
where 
\be\label{choice of delta_1}
\delta_{1}\triangleq \frac{C\,|\Gamma_1|^{\alpha}\,t_{*}^{N_{\alpha}}}{b_1\, N_{\alpha}}, \ee
then (\ref{relation of lamda and M for small time, kth}) through (\ref{choice of delta_1}) together implies that $t_{k}\geq t_{*}$. As a conclusion, in {\bf Step $\bf k$}, as long as (\ref{choice of lam, kth}) has a solution $\lambda_{k}>1$, we define 
$M_{k}$ as in (\ref{notation in kth step}). Then $t_{k}\geq t_{*}$ and therefore $T_{k}\geq kt_{*}$.
\end{itemize}

According to Lemma \ref{Lemma, lower bdd of steps, basic}, if
$1\leq k\leq L$ with 
\[L>\frac{1}{10(q-1)}\Big(\frac{1}{M_{0}^{q-1}\delta_{1}}-9q\Big),\]
then there exists a solution $\lam_{k}>1$ to  (\ref{choice of lam, kth}). So we can construct a finite sequence $\{M_{j}\}_{0\leq j\leq L}$ such that $T_{L}\geq Lt_{*}$. As a result,
\be\label{lower bdd in the proof}
T^{*}>T_{L}\geq Lt_{*}>\frac{1}{10(q-1)}\Big(\frac{1}{M_{0}^{q-1}\delta_{1}}-9q\Big)t_{*}.\ee
Note that when $t_{*}\rightarrow 0^{+}$, $\delta_{1}$ also tends to $0^{+}$. So if we choose $t_{*}$ to be small, then the right hand side of (\ref{lower bdd in the proof}) is positive.

The final question is how to choose $t_{*}\in(0,1]$ to maximize the right hand side of (\ref{lower bdd in the proof}). Plugging (\ref{choice of delta_1}) into (\ref{lower bdd in the proof}), 
\begin{align}
T^{*} &\geq \frac{1}{10(q-1)}\bigg(\frac{b_{1}N_{\alpha}}{M_{0}^{q-1}\,C\, |\Gamma_1|^{\alpha}\,t_{*}^{N_{\alpha}}}-9q\bigg)\,t_{*} \notag\\
&=\frac{9q}{10(q-1)}\bigg(\frac{C_{1}N_{\alpha}}{q\,M_{0}^{q-1}\,|\Gamma_{1}|^{\alpha}}t_{*}^{1-N_{\alpha}}-t_{*}\bigg), \label{lower bdd by t_(*)}
\end{align}
where $C_{1}\triangleq b_{1}/(9C)$ is a constant only depending on $N$ and $\O$. In order to maximize the right hand side of (\ref{lower bdd by t_(*)}), define \[A=\frac{C_{1}N_{\alpha}}{q M_{0}^{q-1}\,|\Gamma_{1}|^{\alpha}}\quad\text{and}\quad \beta=1-N_{\alpha}\in[1/2,1).\] 
Regarding the right hand side of (\ref{lower bdd by t_(*)}) to be a function of $t_{*}$ on $[0,1]$, it follows from Lemma \ref{Lemma, simple ineq} that the maximum of this function is
\[(1-\beta)A\,(\min\{1, \beta A\})^{\beta/(1-\beta)}\]
and this maximum occurs at $t_{*}=(\min\{1, \beta A\})^{1/(1-\beta)}$. As a result,
\begin{align*}
T^{*} &\geq \frac{9q}{10(q-1)}\,(1-\beta)A\,\big(\min\{1,\beta A\}\big)^{\beta/(1-\beta)}.
\end{align*}
Noticing that $\beta\geq 1/2$, so
\begin{align}
T^{*} &\geq \frac{9q}{10(q-1)}\,(1-\beta)A\,\Big(\min\Big\{1,\frac{A}{2}\Big\}\Big)^{\beta/(1-\beta)} \notag\\
&= \frac{9C_{1}N_{\alpha}^{2}}{10(q-1)M_{0}^{q-1}\,|\Gamma_{1}|^{\alpha}}\bigg(\min\bigg\{1,\,\frac{C_{1}N_{\alpha}}{2q M_{0}^{q-1}\,|\Gamma_{1}|^{\alpha}}\bigg\}\bigg)^{\frac{1}{N_{\alpha}}-1} \notag\\
&\geq \frac{C_{2}}{(q-1)M_{0}^{q-1}\,|\Gamma_{1}|^{\alpha}}\,\bigg(\min\bigg\{1,\frac{1}{q M_{0}^{q-1}|\Gamma_{1}|^{\alpha}}\bigg\}\bigg)^{\frac{1}{N_{\alpha}}-1}, \label{lower bdd, complete}
\end{align}
where 
\be\label{const on alpha}
C_{2}=\frac{9C_{1}N_{\alpha}^{2}}{10}\,\Big(\min\Big\{1,\frac{C_{1}N_{\alpha}}{2}\Big\}\Big)^{\frac{1}{N_{\alpha}}-1}\ee
is a constant depending on $N$, $\O$ and $\alpha$. 

In particular, if we choose $\alpha=0$ in (\ref{lower bdd, complete}) and (\ref{const on alpha}), then it follows from $N_{\alpha}=\frac{1-(N-1)\alpha}{2}=\frac{1}{2}$ that
\[T^{*}\geq \frac{C_{3}}{(q-1)M_{0}^{q-1}}\min\bigg\{1,\frac{1}{q M_{0}^{q-1}}\bigg\},\]
where $C_{3}$ is a positive constant only depending on $N$ and $\O$.
\end{proof}

Finally we will prove two technical results, Lemma \ref{Lemma, lower bdd of steps, basic} and Lemma \ref{Lemma, simple ineq},   
which were used in the above proof. In the following, for any $q>1$, we write  
\be\label{const on q}
E_{q}=(q-1)^{q-1}/q^q. \ee
By elementary calculus,
\be\label{est on E_q}
\frac{1}{3q}<E_{q}<\min\Big\{ \frac{1}{q}, \frac{1}{(q-1)\,e}\Big\}<1. \ee
Before stating Lemma \ref{Lemma, lower bdd of steps, basic}, we discuss a simple property as below.

\begin{lemma}\label{Lemma, criteria for step continue}
For any $q>1$, write $E_{q}$ as in (\ref{const on q}) and define $g:(1,\infty)\rightarrow \m{R}$ by
\be\label{auxiliary fcn}
g(\lambda)=\frac{\lambda-1}{\lambda^{q}}.\ee
Then the following two claims hold.
\begin{itemize}
\item[(1)] For any $y\in(0,E_{q}]$, there exists unique $\lambda\in\big(1,\frac{q}{q-1}\big]$ such that $g(\lambda)=y$.
\item[(2)] For any $y>E_{q}$, there does not exist $\lambda>1$ such that $g(\lambda)=y$.
\end{itemize}
\end{lemma}

\begin{proof}
Since $g$ is strictly increasing on the interval $\big(1,\frac{q}{q-1}\big]$ and strictly decreasing on the interval $\big[\frac{q}{q-1}, \infty\big)$, it reaches the maximum at $\lambda=q/(q-1)$. Noticing that
\[g\Big(\frac{q}{q-1}\Big)=\frac{(q-1)^{q-1}}{q^q}=E_{q},\]
then the claims (1) and (2) follow directly. 
\end{proof}

\begin{lemma}\label{Lemma, lower bdd of steps, basic}
Given $q>1$, $M_{0}>0$ and $\delta_{1}>0$. Denote $E_{q}$ as in (\ref{const on q}) and construct a (finite) sequence $\{M_{k}\}_{k\geq 0}$ inductively as follows (note $M_{0}$ has been given). Suppose $M_{k-1}$ has been constructed for some $k\geq 1$, then whether defining $M_{k}$ depends on how large $M_{k-1}$ is.

\begin{itemize}
\item[$\diamond$] If $M_{k-1}^{q-1}\,\delta_{1}\leq E_{q}$, then Lemma \ref{Lemma, criteria for step continue} asserts there exists a unique $\lambda_{k}\in \big(1,\frac{q}{q-1}\big]$ such that 
\be\label{choice of step length}
\frac{\lam_{k}-1}{\lam_{k}^{q}}=M_{k-1}^{q-1}\,\delta_{1}. \ee
Then we define $M_{k}=\lam_{k}M_{k-1}$ and continue to the next step. 

\item[$\diamond$] If $M_{k-1}^{q-1}\,\delta_{1}>E_{q}$, then Lemma \ref{Lemma, criteria for step continue} implies the nonexistence of $\lam_{k}>1$ such that (\ref{choice of step length}) holds. So we stop the construction.  
\end{itemize}
We claim this construction stops in finite steps and if the last term is denoted as $M_{L}$, then   
\be\label{lower bdd of steps, basic} 
L>\frac{1}{10(q-1)}\bigg(\frac{1}{M_{0}^{q-1}\delta_{1}}-9q\bigg).\ee
\end{lemma}

\begin{proof} 
First, we will show the construction has to stop in finite steps. Plugging $\lam_{k}=M_{k}/M_{k-1}$ into (\ref{choice of step length}) and rearranging, we obtain 
\be\label{def of M_(k), lemma}
M_{k}= M_{k-1}+M_{k}^{q}\delta_{1}.\ee
Then it follows from (\ref{def of M_(k), lemma}) that the sequence $\{M_{k}\}$ is strictly increasing. 
Therefore, 
\[M_{k}\geq M_{k-1}+M_{k-1}M_{0}^{q-1}\delta_{1}=\big(1+M_{0}^{q-1}\delta_{1}\big)M_{k-1}.\]
Consequently,
\[M_{k}\geq \big(1+M_{0}^{q-1}\delta_{1}\big)^{k}M_{0}.\]
Hence, $M_{k}^{q-1}\delta_{1}$ will exceed $E_{q}$ when $k$ is sufficiently large. So the construction will stop in finite steps. 

Next suppose the constructed sequence is $\{M_{k}\}_{0\leq k\leq L}$. We will derive the lower bound (\ref{lower bdd of steps, basic}) for $L$ as below. If \[M_{0}^{q-1}\delta_{1}>\frac{1}{9q},\]
then (\ref{lower bdd of steps, basic}) holds automatically since the right hand side of (\ref{lower bdd of steps, basic}) is negative. So we will just assume \[M_{0}^{q-1}\delta_{1}\leq \frac{1}{9q}.\] 
Taking advantage of (\ref{est on E_q}), we know
\[M_{0}^{q-1}\delta_{1}\leq \min\{1/2, E_{q}\}.\]
In addition, since the last term of the constructed sequence is $M_{L}$, then $M_{L}^{q-1} \delta_{1}> E_{q}$. As a result, $L\geq 1$ and there exists an index $L_{0}\in [1,L]$ such that
\be\label{L_0}
M_{L_{0}-1}^{q-1}\delta_{1}\leq \min\{1/2, E_{q}\}\quad\text{and}\quad M_{L_{0}}^{q-1}\delta_{1}>\min\{1/2, E_{q}\}.\ee
The reason of considering $\min\{1/2, E_{q}\}$ here instead of $E_{q}$ is because later we need the upper bound $1/2$ to justify (\ref{quadratic ineq for x_k}). 

According to (\ref{def of M_(k), lemma}), 
\[M_{k-1}=M_{k}-M_{k}^{q} \delta_{1}=M_{k}\big(1-M_{k}^{q-1} \delta_{1}\big).\] 
Raising both sides to the power $q-1$ and multiplying by $\delta_{1}$,
\[M_{k-1}^{q-1}\delta_{1}=M_{k}^{q-1}\big(1-M_{k}^{q-1} \delta_{1}\big)^{q-1}\delta_{1}.\]
Define $x_{k}=M_{k}^{q-1}\delta_{1}$. Then 
\be\label{iteration for x_k}
x_{k-1}=x_{k}\,(1-x_{k})^{q-1}, \quad\forall\, 1\leq k\leq L_{0}.\ee
Moreover, $$x_0=M_{0}^{q-1}\delta_{1},\quad x_{L_{0}-1}\leq \min\{1/2, E_{q}\} \quad\text{and}\quad x_{L_{0}}>\min\{1/2, E_{q}\}.$$
Noticing that $M_{L_{0}}=\lam_{L_{0}}M_{L_{0}-1}\leq \frac{q}{q-1}M_{L_{0}-1}$, so
\[x_{L_{0}}=\bigg(\frac{M_{L_{0}}}{M_{L_{0}-1}}\bigg)^{q-1}x_{L_{0}-1}\leq \Big(\frac{q}{q-1}\Big)^{q-1}E_{q}=\frac{1}{q}.\]

Now we claim the following inequality:
\begin{align}\label{ineq between last and first for x_k}
\frac{1}{x_{0}} \leq \frac{1}{x_{L_{0}-1}}+10(q-1)(L_{0}-1).
\end{align}
In fact, if $L_{0}=1$, then (\ref{ineq between last and first for x_k}) automatically holds. If $L_{0}\geq 2$, then for any $1\leq k\leq L_{0}-1$, 
\be\label{quadratic ineq for x_k}
x_{k-1}=x_{k}(1-x_{k})^{q-1}\geq x_{k}\big(1-2(q-1)x_{k}\big)\ee
since $0<x_{k}\leq x_{L_{0}-1}\leq 1/2$. Recalling the fact $x_{k}\leq x_{L_{0}-1}\leq E_{q}$ and the estimate $E_{q}<\frac{1}{(q-1)e}$ in (\ref{est on E_q}), we have
\[1-2(q-1)x_{k}\geq 1-2(q-1)E_{q}\geq \frac{1}{5}.\]
Hence, taking the reciprocal in (\ref{quadratic ineq for x_k}) yields 
\begin{align}
\frac{1}{x_{k-1}} &\leq \frac{1}{x_{k}\big[1-2(q-1)x_{k}\big]} \notag\\
& =\frac{1}{x_{k}}+\frac{2(q-1)}{1-2(q-1)x_{k}} \notag\\
&\leq \frac{1}{x_{k}}+10(q-1). \label{iteration ineq}
\end{align}
Summing up (\ref{iteration ineq}) for $k$ from $1$ to $L_{0}-1$, we also obtain (\ref{ineq between last and first for x_k}).  

Finally, since
\[\frac{1}{3q}<\min\Big\{\frac{1}{2}, E_{q}\Big\}\leq x_{L_{0}}\leq \frac{1}{q},\]
then 
\begin{align*}
x_{L_{0}-1} &= x_{L_{0}}(1-x_{L_{0}})^{q-1}\\
& >\frac{1}{3q}\Big(1-\frac{1}{q}\Big)^{q-1}=\frac{E_{q}}{3}.
\end{align*}
Plugging the above inequality and $x_{0}=M_{0}^{q-1}\delta_{1}$ into (\ref{ineq between last and first for x_k}), it gives
\begin{align*}
\frac{1}{M_{0}^{q-1}\delta_{1}} &< \frac{3}{E_{q}}+10(q-1)(L_{0}-1)\\
&< 9q+10(q-1)(L_{0}-1).
\end{align*}
Rearranging this inequality yields
\[L_{0}>\frac{1}{10(q-1)}\bigg(\frac{1}{M_{0}^{q-1}\delta_{1}}-9q\bigg)+1.\]
Hence,
\[L\geq L_{0}>\frac{1}{10(q-1)}\bigg(\frac{1}{M_{0}^{q-1}\delta_{1}}-9q\bigg).\]

\end{proof}

\begin{lemma}\label{Lemma, simple ineq}
Fix any two constants $A>0$ and $\beta\in(0,1)$. Define $f:[0,1]\rightarrow \m{R}$ by $f(t)=A\,t^{\beta}-t$. Then $f$ attains the maximum at $t=(\min\{1, \beta A\})^{1/(1-\beta)}$ and 
\be\label{min of auxillary fcn}\max_{0\leq t\leq 1}f(t) \geq (1-\beta)A\,(\min\{1, \beta A\})^{\beta/(1-\beta)}.\ee
\end{lemma}
\begin{proof}
For any $t\in(0,1]$, $f'(t)=\beta A\,t^{\beta-1}-1$, so 
$f$ is increasing when $0\leq t\leq \min\{1, (\beta A)^{1/(1-\beta)}\}$ and decreasing when $\min\{1, (\beta A)^{1/(1-\beta)}\}\leq t\leq 1$. Hence, $f$ attains its maximum at $t=\min\{1, (\beta A)^{1/(1-\beta)}\}$. In addition,
\begin{itemize}
\item If $\beta A\geq 1$, then 
\[\max_{0\leq t\leq 1}f(t)=f(1)=A-1\geq (1-\beta)\,A.\]

\item If $0<\beta A< 1$, then 
\begin{align*}
\max_{0\leq t\leq 1}f(t) &=f\big[(\beta A)^{1/(1-\beta)}\big]\\
&=A\,(\beta A)^{\beta/(1-\beta)}-(\beta A)^{1/(1-\beta)}\\
&=(1-\beta)A\,(\beta A)^{\beta/(1-\beta)}.
\end{align*}
\end{itemize}
Combining these two cases, (\ref{min of auxillary fcn}) is justified.
\end{proof}

\appendix
\section{Time-shifted representation formula}
\label{Sec, adapted rep formula}
In Corollary 3.9 of \cite{YZ16}, it derived the following representation formula (\ref{rep formula}) for the solution $u$ to (\ref{Prob}) (Note $(D\Phi)(x-y,t-\tau)$ in that Corollary is equal to $-D_{y}[\Phi(x-y,t-\tau)]$ in (\ref{rep formula})). Namely, if $T^{*}$ denotes the maximal existence time of $u$, then for any boundary point $(x,t)\in\p\O\times[0,T^{*})$, 
\begin{align}
u(x,t) = &\,\, 2\int_{\O}\Phi(x-y,t)\,u_{0}(y)\,dy- 
2\int_{0}^{t}\int_{\p\O}D_{y}\big[\Phi(x-y,t-\tau)\big]\cdot \ora{n}(y)\,u(y,\tau)\,dS(y)\,d\tau \notag\\
&+2\int_{0}^{t}\int_{\Gamma_1}\Phi(x-y,t-\tau)\,u^{q}(y,\tau)\,dS(y)\,d\tau.  \label{rep formula}
\end{align} 
We want to remark that there is also representation formula for the inside point $(x,t)\in\O\times[0,T^{*})$ (see Theorem 3.8 in \cite{YZ16}), but that formula is different from (\ref{rep formula}) in that there does not have the coefficients 2 in front of the integrals on the right hand side. The appearance of the coefficient 2 in (\ref{rep formula}) is due to the jump relation of the single-layer heat potential when $x\in\p\O$ (see e.g. \cite{Fri64}, Sec. 2, Chap. 5). 

The formula in (\ref{rep formula}) based on the initial data $u_{0}(\cdot)$. Now for any $T\in(0,T^{*})$, we are asking that if regarding $T$ to be the initial time and $u(\cdot,T)$ to be the initial data, then do we still have the representation formula? It seems trivial by just shifting the time $T$, but we should be careful since Corollary 3.9 in \cite{YZ16} deals with $C^{1}(\ol{\O})$ initial data $u_{0}$ but $u(\cdot,T)$ is only in $C^{2}(\O)\cap C(\ol{\O})$. The next lemma claims that as long as $u$ is the solution to (\ref{Prob}) with the assumption (\ref{assumption on prob}), then for any $T\in[0,T^{*})$, there also holds a representation formula similar to (\ref{rep formula}) but with initial data $u(\cdot,T)$.

\begin{lemma}
Assume (\ref{assumption on prob}). Let $T^{*}$ be the maximal existence time and $u$ be the maximal solution to (\ref{Prob}). Then for any $x\in\p\O$, $T\in [0,T^{*})$ and $t\in [0,T^{*}-T)$,
\begin{align}
u(x,T+t) = &\,\, 2\int_{\O}\Phi(x-y,t)\,u(y,T)\,dy-2\int_{0}^{t}\int_{\p\O}D_{y}\big[\Phi(x-y,t-\tau)\big]\cdot \ora{n}(y)\,u(y,T+\tau)\,dS(y)\,d\tau \notag\\
&+2\int_{0}^{t}\int_{\Gamma_1}\Phi(x-y,t-\tau)\,u^{q}(y,T+\tau)\,dS(y)\,d\tau.  \label{rep formula from any time}
\end{align}
\end{lemma}
\begin{proof}
When $T=0$, (\ref{rep formula from any time}) is just the representation formula (\ref{rep formula}) which has been proven in \cite{YZ16}. So we can assume $T>0$. Define $v:\ol{\O}\times[0,T^{*}-T)\rightarrow \m{R}$ by 
\be\label{def of v}
v(x,t)=u(x,T+t).\ee 
Then $v\in C^{2,1}\big(\O\times(0,T^{*}-T)\big)\bigcap C\big(\ol{\O}\times[0,T^{*}-T)\big)$ and satisfies 
\be\label{Prob, from any time}
\left\{\begin{array}{lll}
v_{t}(x,t)=\Delta v(x,t) &\text{in}& \Omega\times (0,T^{*}-T),\\
\frac{\partial v}{\partial n}(x,t)=u^{q}(x,T+t) &\text{on}& \Gamma_1\times (0,T^{*}-T),\\
\frac{\partial v}{\partial n}(x,t)=0 &\text{on}& \Gamma_2\times (0,T^{*}-T),\\
v(x,0)=u(x,T) &\text{in}& \Omega.
\end{array}\right.\ee

Since $u(\cdot,T)$ is continuous on $\ol{\O}$ as a function in space variable $x$, we can continuously extend $u(\cdot,T)$ to $\m{R}^{N}$ and still denote it to be $u(\cdot, T)$. Let $\eta\in C_{0}^{\infty}(\m{R}^{N})$ be the standard mollifier. That is, 
\[\eta(x)=\left\{\begin{array}{lll}
C\exp\big(\frac{1}{|x|^{2}-1}\big) & \text{if} & |x|<1,\\
0 & \text{if} & |x|\geq 1,
\end{array}\right.\] 
where the positive constant $C$ is selected so that $\int_{\m{R}^{N}}\eta(x)\,dx=1$. Then for any $j\geq 1$, there exists $\eps_{j}>0$ such that by defining 
\[\eta_{j}(x)=\eps_{j}^{-N}\eta(x/\eps_{j}) \quad \text{and}\quad g_{j}(x)=\big(\eta_{j}(\cdot)*u(\cdot,T)\big)(x),\] then
\be\label{approx to initial data}
\max_{x\in\ol{\O}}|g_j(x)-u(x,T)|\leq \frac{1}{j}. \ee
Since $g_{j}\in C^{1}(\ol{\O})$, it follows from Theorem B.4 in \cite{YZ16} that there exists $v_{j}\in C^{2,1}\big(\O\times(0,T^{*}-T)\big)\bigcap C\big(\ol{\O}\times[0,T^{*}-T)\big)$ such that 
\be\label{Prob, approx from any time}
\left\{\begin{array}{lll}
(v_{j})_{t}(x,t)=\Delta v_{j}(x,t) &\text{in}& \Omega\times (0,T^{*}-T), \vspace{0.02in}\\ 
\frac{\partial v_{j}}{\partial n}(x,t)=u^{q}(x,T+t) &\text{on}& \Gamma_1\times (0,T^{*}-T), \vspace{0.02in}\\
\frac{\partial v_{j}}{\partial n}(x,t)=0 &\text{on}& \Gamma_2\times (0,T^{*}-T), \vspace{0.02in}\\
v_{j}(x,0)=g_{j}(x) &\text{in}& \Omega.
\end{array}\right.\ee
Again due to the fact that $g_{j}\in C^{1}(\ol{\O})$, we can apply the representation formula (\ref{rep formula}) to $v_{j}$ so that for any $(x,t)\in\p\O\times[0,T^{*}-T)$,
\begin{align}\label{rep formula for approx soln}
v_{j}(x,t) = &\,\, 2\int_{\O}\Phi(x-y,t)\,g_{j}(y)\,dy- 
2\int_{0}^{t}\int_{\p\O}D_{y}\big[\Phi(x-y,t-\tau)\big]\cdot \ora{n}(y)\,v_{j}(y,\tau)\,dS(y)\,d\tau \notag\\
&+2\int_{0}^{t}\int_{\Gamma_1}\Phi(x-y,t-\tau)\,u^{q}(y,T+\tau)\,dS(y)\,d\tau.  
\end{align} 

Define $w_{j}=v_{j}-v$. Then $w_{j}\in C^{2,1}\big(\O\times(0,T^{*}-T)\big)\bigcap C\big(\ol{\O}\times[0,T^{*}-T)\big)$ and satisfies
\bes \left\{\begin{array}{lll}
(w_{j})_{t}(x,t)=\Delta w_{j}(x,t) &\text{in}& \Omega\times (0,T^{*}-T), \vspace{0.02in}\\
\frac{\partial w_{j}}{\partial n}(x,t)=0 &\text{on}& \Gamma_1\times (0,T^{*}-T), \vspace{0.02in}\\
\frac{\partial w_{j}}{\partial n}(x,t)=0 &\text{on}& \Gamma_2\times (0,T^{*}-T), \vspace{0.02in}\\
w_{j}(x,0)=g_{j}(x)-u(x,T) &\text{in}& \Omega.
\end{array}\right.\ees
So it follows from the maximum principle and the Hopf lemma that for any $(x,t)\in\ol{\O}\times[0,T^{*}-T)$,
$$|w_{j}(x,t)|\leq \max_{x\in\ol{\O}}|g_j(x)-u(x,T)|\leq \frac{1}{j}.$$
That is,
\be\label{pt conv for v_j}
|v_{j}(x,t)-v(x,t)|\leq \frac{1}{j}, \quad\forall\,(x,t)\in\ol{\O}\times[0,T^{*}-T).\ee

Now fixing any point $(x,t)\in\p\O\times[0,T^{*}-T)$ and sending $j\rightarrow\infty$ in (\ref{rep formula for approx soln}), then it follows from (\ref{pt conv for v_j}), (\ref{approx to initial data}) and Lebesgue's dominated convergence theorem that 
\begin{align*}
v(x,t) = &\,\, 2\int_{\O}\Phi(x-y,t)\,u(y,T)\,dy- 
2\int_{0}^{t}\int_{\p\O}D_{y}\big[\Phi(x-y,t-\tau)\big]\cdot \ora{n}(y)\,v(y,\tau)\,dS(y)\,d\tau \notag\\
&+2\int_{0}^{t}\int_{\Gamma_1}\Phi(x-y,t-\tau)\,u^{q}(y,T+\tau)\,dS(y)\,d\tau.  
\end{align*} 
Finally, (\ref{rep formula from any time}) is verified by recalling the definition (\ref{def of v}) for $v$.
\end{proof}

\bibliographystyle{plain}
\bibliography{References}

\end{document}